\documentclass[11pt]{article}

\usepackage[T1]{fontenc}
\usepackage[utf8]{inputenc}
\usepackage{lmodern}
\usepackage{microtype}
\usepackage{amsmath,amssymb,amsthm,mathtools,mleftright}
\usepackage{xcolor}
\usepackage{tikz}
\usepackage{pgfplots}
\pgfplotsset{compat=1.18}
\usepackage{hyperref}
\usepackage[nameinlink,capitalise]{cleveref}

\hypersetup{
  colorlinks=true,
  linkcolor=blue,
  citecolor=blue,
  urlcolor=blue
}

\newtheorem{theorem}{Theorem}[section]
\newtheorem{proposition}[theorem]{Proposition}
\newtheorem{corollary}[theorem]{Corollary}
\newtheorem{lemma}[theorem]{Lemma}
\newtheorem{remark}[theorem]{Remark}
\newtheorem{definition}[theorem]{Definition}
\newtheorem*{remark*}{Remark}
\newtheorem*{theorem*}{Theorem}

\newcommand{\R}{\mathbb{R}}
\newcommand{\SL}{\operatorname{SL}}
\newcommand{\GLp}{\operatorname{GL}^{+}}
\newcommand{\SO}{\operatorname{SO}}
\newcommand{\Cof}{\operatorname{Cof}}
\newcommand{\arcosh}{\operatorname{arcosh}}
\newcommand{\arsinh}{\operatorname{arcsinh}}
\newcommand{\sgn}{\operatorname{sgn}}

\makeatletter
\let\@fnsymbol\@arabic
\makeatother

\usepackage{bbm}
\newcommand{\id}{{\boldsymbol{\mathbbm{1}}}}

\usepackage{subcaption}

\numberwithin{equation}{section}
\title{Polyconvexity for incompressible inversion-symmetric energies of Valanis--Landel type}
\author{Ionel-Dumitrel Ghiba\thanks{Corresponding author: 
		Ionel-Dumitrel Ghiba, \ \  Department of Mathematics, Alexandru Ioan Cuza University of Ia\c si,  Blvd.
		Carol I, no. 11, 700506 Ia\c si,
		Romania; and  Octav Mayer Institute of Mathematics of the
		Romanian Academy, Ia\c si Branch,  700505 Ia\c si, email:  dumitrel.ghiba@uaic.ro}, \quad Maximilian P. Wollner\thanks{Maximilian P. Wollner,  \ \   Institute of Biomechanics, Graz University of Technology, Stremayrgasse 16/2, 8010, Graz, Austria, email: wollner@tugraz.at} 
	\quad and   \quad   Patrizio Neff\,\thanks{Patrizio Neff,  \  \ Head of Lehrstuhl f\"{u}r Nichtlineare Analysis und Modellierung, Fakult\"{a}t f\"{u}r
		Mathematik, Universit\"{a}t Duisburg-Essen,  Thea-Leymann Str. 9, 45127 Essen, Germany, email: patrizio.neff@uni-due.de}
}
\date{}

\begin{document}
\maketitle

\begin{abstract}
\noindent {Let $\lambda_i = \nu_i(F)$ denote the three singular values of
the deformation gradient $F \in \GLp(3)$.} We consider the family of incompressible isotropic energies
$
  W_\psi(F)=\sum_i \psi\mleft(|\!\log\lambda_i|\mright)$ with $\psi\colon[0,\infty)\mapsto\R$.
Set
$g(s)=\psi\mleft(\arcosh\frac{s}{2}\mright)$ for all $s\geq2$.
{If $g$ is convex and non-decreasing on $[2,\infty)$} {and $\bar g$ denotes any convex and non-decreasing extension of $g$ to $[0,\infty)$ (e.g., $\bar g(s)=g(\max\{s,2\})$)},
then $W_\psi$ is the restriction to $\SL(3)$ of the explicit polyconvex function \ {$
  F\mapsto\sum_i \bar g\mleft(\nu_i(F+\Cof F)\mright).
$}
The proof uses the identity $\nu_i(F+\Cof F)=\lambda_i+\lambda_i^{-1}$, up to permutation,
on $\SL(3)$ and Ball's convexity theorem for functions of the singular values.
We also give a direct proof along rank-one lines contained in $\SL(3)$ and
derive a convenient one-dimensional differential sufficient condition. In
particular,
$
  F\mapsto\sum_i e^{\log^2\!\lambda_i}
$
is rank-one convex on $\SL(3)$ and possesses the stated polyconvex extension.
A simple-shear computation shows that scalar convexity in $\log \lambda_i$ alone is
insufficient; the quadratic Hencky energy $\sum_i \log^2\lambda_i$ is not
rank-one convex on $\SL(3)$.
\end{abstract}


\section{Introduction}

Legendre--Hadamard ellipticity is widely viewed  as a
necessary constitutive requirement for material stability; see, e.g.,
\cite{ZubovRudev,Antman95}.
However, its verification for a given elastic energy is burdensome and, in
most cases, is not amenable to a direct analytical check. This statement still applies to the  incompressible case \cite{ZubovRudev,aubert1995necessary,Ball1976}. A stronger and often more tractable condition is Ball's polyconvexity, which implies rank-one convexity and, for sufficiently smooth energies, Legendre--Hadamard ellipticity 
\cite{Ball1976}.

In this note, we investigate elastic energies of Valanis--Landel type. The Valanis--Landel hypothesis represents an isotropic strain energy as a
separable symmetric function of the principal stretches; see
\cite{valanis1967strain,Ogden83}.  Many energy formulae on the incompressibility constraint are of this form, e.g., Mooney--Rivlin energy, Neo-Hooke energy, Ogden energy, etc. More precisely, an isotropic energy $W\colon\GLp(3)\mapsto\R$ is said to be of Valanis--Landel type if
there exists a scalar function $w\colon(0,\infty)\mapsto\R$ such that
\begin{equation}\label{eq:VL}
  W(F)=\sum_{i=1}^3 w\mleft(\lambda_i\mright),
\end{equation}
where
\begin{equation}\label{eq:lambdas}
  \lambda_1=\nu_1(F)>0,\qquad \lambda_2=\nu_2(F)>0,\qquad\text{and}\qquad\lambda_3=\nu_3(F)>0
\end{equation}
denote the principal stretches, i.e., the singular values of $F$. Singular values are always counted with multiplicity; unless an ordering is explicitly imposed, identities
involving the index $i$ are understood up to permutation.

Valanis{--}Landel originally proposed for moderate stretches $w(\lambda)=2\mu\, \lambda\,(\log \lambda-1)$ which corresponds to
\begin{equation}\label{eq:VLo}
  W(F)=2\mu \langle U, \log U-\id_3 \rangle,
\end{equation} 
an energy which has  already been  proposed by G.F. Becker in 1893 (see \cite{NBecker}).

Since  the Valanis{--}Landel form is intuitive, simple and separable, it
might suggest that Legendre--Hadamard ellipticity on $\SL(3)$ also takes
a simple form. For example, one may consider the Biot-type energy
\begin{equation}
  W_{\rm Biot}(F)=\|U-\id_3\|^2=\sum_{i=1}^3\mleft(\lambda_i-1\mright)^2\qquad\text{with}\qquad w(\lambda)=(\lambda-1)^2.
\end{equation}

The Valanis{--}Landel  class includes many standard models in
nonlinear elasticity and has also been used in the construction of polyconvex
extensions of logarithmic strain energies; see \cite{MartinGhibaNeff2019}.
Further developments, interpretations and applications of the separable
Valanis--Landel representation can be found in
\cite{PengLandel1972,Landel1998,Rivlin2003,rivlin2006relation,Valanis2022,SussmanBathe2009}.

However, the simple convexity of $w$ with respect to the stretch variable alone is not
sufficient for Legendre--Hadamard ellipticity, already in the planar case; see
\cite{MartinEtAlBiot2025}.
Thus, our overarching aim is to obtain tractable conditions on the scalar
function $w$ which ensure Legendre--Hadamard ellipticity. In order to
facilitate the approach, we additionally impose the structural condition of term-wise inversion symmetry, i.e.,
\begin{equation}\label{eq:inversion-w}
  w(\lambda)=w(\lambda^{-1})\qquad\forall\,\lambda>0{.}
\end{equation}
For a recent systematic treatment of inversion symmetry in separable line,
area and volume elastic energies, see
\cite{VitucciTrentadueDeTommasi2026}.
Equivalently, one may always write
\begin{equation}\label{eq:wpsi}
  w(\lambda)=\psi(|\!\log\lambda|)\qquad\text{and}\qquad \psi(x)=w(e^x)\qquad\forall\,x\geq0.
\end{equation}

For comparison, the incompressible Mooney--Rivlin energy
\cite{mooney1940theory,rivlin1948large} can be written on $\SL(3)$ as
\begin{equation}\label{eq:MR}
  W_{\mathrm{MR}}(F)
  =\alpha\,\lVert F\rVert^2+\beta\,\lVert\Cof F\rVert^2=\alpha\sum_{i=1}^3\lambda_i^2
   +\beta\sum_{i=1}^3\lambda_i^{-2}\qquad\forall\,\alpha,\beta\geq0.
\end{equation}
It is both of Valanis--Landel type and polyconvex. Its scalar summand is
term-wise inversion-symmetric precisely when $\alpha=\beta$.

The purpose of this note is to exploit the identity
$\lambda+\lambda^{-1}=2\cosh(|\!\log\lambda|)$ on the incompressibility
constraint ${\rm SL}(3)$. It gives an explicit polyconvex extension for the family of incompressible isotropic energies
\begin{align}\label{eq:abstract-energy}
W_\psi(F)=\sum_{i=1}^3\psi\mleft(|\!\log\lambda_i|\mright)\qquad\text{with}\qquad\psi\colon[0,\infty)\mapsto\R,
\end{align} a direct verification along admissible rank-one
lines, and a one-dimensional differential sufficient condition.
For $F\in\GLp(3)$, let
\begin{equation}\label{eq:polar}
  F=R\,U,\qquad R\in\SO(3),\qquad\text{and}\qquad
  U=\sqrt{F^TF}\in\operatorname{Sym}^{++}(3)
\end{equation}
be the right polar decomposition (see \cite{agn_neff2014grioli}). The eigenvalues of $U$ are the singular
values of $F$.

We shall use the following classical result in precisely the form needed below.

\begin{theorem}[Ball's singular-value convexity theorem]\label{thm:ball}
Let $G\colon[0,\infty)^3\mapsto\R$ be symmetric, convex, and non-decreasing in
each variable. Then
$
  X\mapsto G\mleft(\nu_1(X),\nu_2(X),\nu_3(X)\mright)
$
is convex on $\R^{3\times3}$.
\end{theorem}

This is the sufficiency part of \cite[Theorem~5.1(ii)]{Ball1976}. In particular,
if $\bar g\colon[0,\infty)\mapsto\R$ is convex and non-decreasing, then
\begin{equation}\label{eq:G}
  G(X)=\sum_{i=1}^3\bar g\mleft(\nu_i(X)\mright)
\end{equation}
is convex on $\R^{3\times3}$.

\begin{definition}\label{def:r1}
A function $W\colon\SL(3)\mapsto\R$ is rank-one convex on $\SL(3)$ \ {if the function $t\mapsto W(F_t)$ is convex for all $F_t=F+t\,a\otimes b\in\SL(3)$, where $a,b\in\R^3$,
$F\in\SL(3)$, and $t \in [0,1]$.}
\end{definition}

For a rank-one perturbation, the matrix determinant lemma gives
\begin{equation}\label{eq:detlemma}
  \det(F+t\,a\otimes b)
  =\det F\mleft(1+t\,\langle b, F^{-1}a\rangle \mright).
\end{equation}
Thus, for $F\in\SL(3)$, the entire line $F+t\,a\otimes b$ lies in $\SL(3)$
if and only if
$
 \langle b, F^{-1}a\rangle=0.
$
Since $\SL(3)$ is not convex, rank-one convexity is understood only along
rank-one segments contained in the constraint manifold.

\section{The polyconvex extension}\setcounter{equation}{0}
\subsection{The general case}
\begin{lemma}\label{lem:key}
Let $F\in\SL(3)$ and let $\lambda_i$ be its singular
values. Then, up to permutation,
\begin{equation}\label{eq:key}
  \nu_i(F+\Cof F)=\lambda_i+\lambda_i^{-1}\qquad\forall\, i\in\{1,2,3\}.
\end{equation}
\end{lemma}

\begin{proof}
For $F=R\,U\in\SL(3)$, one has
\begin{equation}\label{eq:cof-polar}
  \Cof F=(\det F)F^{-T}=F^{-T}=R\,U^{-1}.
\end{equation}
Consequently,
\begin{equation}\label{eq:Fcof-polar}
  F+\Cof F=R\,(U+U^{-1}).
\end{equation}
Left multiplication by $R\in\SO(3)$ does not change singular values. Since
$U+U^{-1}$ is symmetric positive definite and has the same eigenvectors as
$U$, its eigenvalues and, hence, its singular values are
$\lambda_i+\lambda_i^{-1}$. This proves \eqref{eq:key}.
\end{proof}

\begin{theorem}[Explicit \ {polyconvex} extension]\label{thm:main}
Let $\psi:[0,\infty)\mapsto\R$ and define
\begin{equation}\label{eq:gdef2}
  g(s)=\psi\mleft(\arcosh\frac{s}{2}\mright)\qquad\forall\,s\geq2.
\end{equation}
Assume that {$g$ is convex and non-decreasing on $[2,\infty)$, and let
$\bar g\colon[0,\infty)\mapsto\R$ be a convex and non-decreasing function such that $\bar g=g$ on $[2,\infty)$.\footnote{{Such an extension always exists; for instance, $\bar g(s)=g(\max\{s,2\})$, i.e., $\bar g(s)=g(2)$ for $0\leq s\leq2$ and $\bar g(s)=g(s)$ for $s\geq2$, is convex and non-decreasing, being the composition of the convex and non-decreasing function $g$ with the convex function $s\mapsto\max\{s,2\}$.}}} Then,
\begin{equation}\label{eq:Wpsi}
  W_\psi(F)=\sum_{i=1}^3\psi\mleft(|\!\log\lambda_i|\mright)\qquad\forall\,F\in\SL(3)
\end{equation}
is the restriction to $\SL(3)$ of the polyconvex function\ {
\begin{equation}\label{eq:Wtilde}
  {\widetilde W_\psi\colon\GLp(3)\mapsto\R,\qquad \widetilde W_\psi(F)= P(F,\Cof F,\det F)=\sum_{i=1}^3\bar g\mleft(\nu_i(F+\Cof F)\mright),}
\end{equation}
{where $P\colon\R^{3\times3}\times\R^{3\times3}\times(0,\infty)\mapsto\R$ is the convex function
\begin{equation}\label{eq:P-def}
  P(X,Y,d) = \sum_{i=1}^3\bar g\mleft(\nu_i(X+Y)\mright).
\end{equation}}
In particular, $W_\psi$ is rank-one convex on $\SL(3)$.}
\end{theorem}
\begin{figure}[ht]
\centering

\begin{subfigure}[t]{0.48\textwidth}
\centering
\begin{tikzpicture}
\begin{axis}[
  width=\textwidth,height=7.2cm,
  domain=0:2.1,samples=120,
  axis lines=middle,
  xmin=0,xmax=2.2,ymin=1.8,ymax=8.5,
  xlabel={$x$},ylabel={$s=2\cosh x$},
  tick label style={black},
  label style={black},
  axis line style={black}]
\addplot[black,thick] {exp(x)+exp(-x)};
\end{axis}
\end{tikzpicture}
\caption{The change of variables $s=2\cosh x$.}
\label{fig:cosh-map}
\end{subfigure}
\hfill
\begin{subfigure}[t]{0.48\textwidth}
\centering
\begin{tikzpicture}
\begin{axis}[
  width=\textwidth,height=7.2cm,
  domain=2:8,samples=120,
  axis lines=middle,
  xmin=1.8,xmax=8.2,ymin=0,ymax=2.2,
  xlabel={$s$},ylabel={$x=\arcosh(s/2)$},
  tick label style={black},
  label style={black},
  axis line style={black}]
\addplot[black,thick] {ln(x/2+sqrt((x/2)^2-1))};
\end{axis}
\end{tikzpicture}
\caption{The inverse change of variables $x=\arcosh(s/2)$.}
\label{fig:arcosh-map}
\end{subfigure}

\caption{The mutually inverse change of variables
$s=2\cosh x$ and $x=\arcosh(s/2)$ between $[0,\infty)$ and $[2,\infty)$.}
\label{fig:cosh-arcosh}
\end{figure}
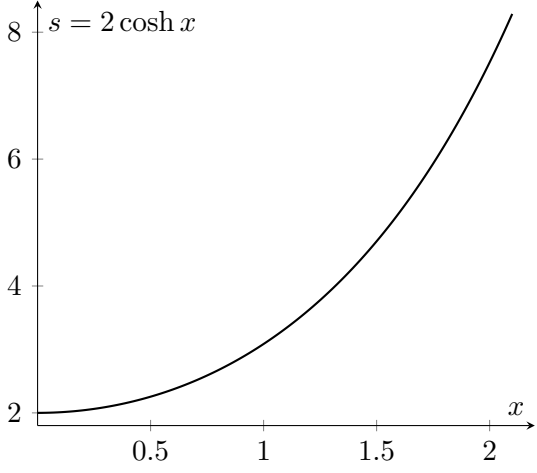
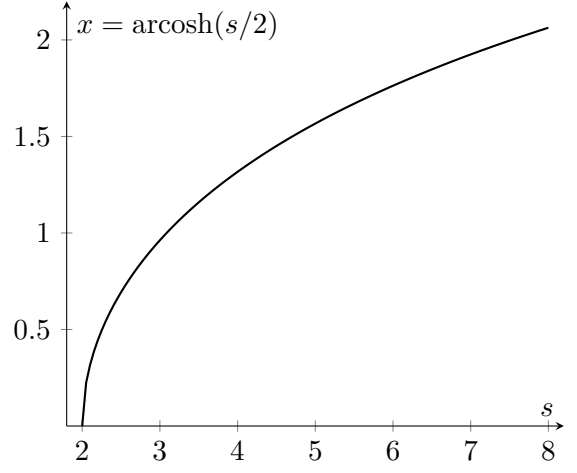
\begin{proof}[{First proof of Theorem~\ref{thm:main} (via Ball's singular-value convexity theorem)}]
Recall the function $G$ defined in \eqref{eq:G}, i.e.,  $G(X)=\sum_i \bar g\mleft(\nu_i(X)\mright)$.  By \cref{thm:ball},
$G$ is convex. {The function $P$ defined in \eqref{eq:P-def} satisfies}
\begin{equation}\label{eq:P}
  P(X,Y,d)=G(X+Y){\qquad\forall\,(X,Y,d)\in\R^{3\times3}\times\R^{3\times3}\times(0,\infty).}
\end{equation}
The map $(X,Y,d)\mapsto X+Y$ is affine, so $P$ is convex. {Hence, by definition of polyconvexity, $\widetilde W_\psi(F)=P(F,\Cof F,\det F)$ is polyconvex on $\GLp(3)$.} \ {Moreover, for all $F \in \SL(3)$, we {have}
\begin{equation}
\label{eq:W-P}
\begin{split}
  {\widetilde W_\psi(F)} = P(F,\Cof F,\det F) &= G(F + \Cof F) \\
  &= \sum_{i=1}^3\bar g\mleft(\nu_i(F+\Cof F)\mright) \\
  &= \sum_{i=1}^3 g\mleft(\nu_i(F+\Cof F)\mright) \\
  &= \sum_{i=1}^3\psi\mleft(\arcosh\frac{\lambda_i+\lambda_i^{-1}}{2}\mright) \\
  &=\sum_{i=1}^3\psi\mleft(|\!\log\lambda_i|\mright),
\end{split}
\end{equation}
since, by {Lemma~\ref{lem:key}},
\begin{equation}\label{eq:sigma-cosh}
  \nu_i(F+\Cof F)=\lambda_i+\lambda_i^{-1}
  =2\cosh(|\!\log\lambda_i|)\geq2
\end{equation}
and $\bar g=g$ on $[2,\infty)$. Thus, {$\widetilde W_\psi|_{\SL(3)}=W_\psi$}. Finally, a polyconvex function is rank-one convex; hence, its restriction is
convex along every rank-one segment contained in $\SL(3)$.}
\end{proof}

{Theorem~\ref{thm:main} can also be proved by means of the characterization of polyconvex isotropic functions in terms of signed singular values due to Wiedemann and Peter \cite{WiedemannPeter2026}; a specialization of this characterization to incompressible energies is given by Geuken et al.\ \cite{Geuken2026}. Every $F\in\R^{3\times3}$ admits a signed singular value decomposition $F=Q_1\operatorname{diag}(\nu_1,\nu_2,\nu_3)\,Q_2$ with $Q_1,Q_2\in\SO(3)$, where $\nu=(\nu_1,\nu_2,\nu_3)\in\R^3$ satisfies $|\nu_i|=\lambda_i$, $i=1,2,3$, and $\nu_1\nu_2\nu_3=\det F$; the triple $\nu$ is called a triple of signed singular values of $F$ and is unique only up to permutations and to a change of sign of two of its entries. A function $W\colon\R^{3\times3}\mapsto\R$ is isotropic if $W(Q_1FQ_2)=W(F)$ for all $F\in\R^{3\times3}$ and all $Q_1,Q_2\in\SO(3)$; such a function is completely determined by
\begin{equation}\label{eq:Phi-ssv}
  \Phi\colon\R^3\mapsto\R,\qquad \Phi(\nu)\coloneq W(\operatorname{diag}(\nu_1,\nu_2,\nu_3)).
\end{equation}
Let
\begin{equation}\label{eq:m-ssv}
  m\colon\R^3\mapsto\R^7,\qquad m(\nu)=(\nu_1,\nu_2,\nu_3,\nu_2\nu_3,\nu_1\nu_3,\nu_1\nu_2,\nu_1\nu_2\nu_3),
\end{equation}
which collects the minors of $\operatorname{diag}(\nu_1,\nu_2,\nu_3)$, in the same way as $(F,\Cof F,\det F)$ collects the minors of $F$. For finite-valued functions, the result of Wiedemann and Peter reads as follows.

\begin{theorem*}[{\cite[Theorem~1.4]{WiedemannPeter2026}}]
Let $W\colon\R^{3\times3}\mapsto\R$ be isotropic and let $\Phi$ be given by \eqref{eq:Phi-ssv}. Then $W$ is lower semicontinuous and polyconvex if and only if there exists a lower semicontinuous convex function $\mathcal G\colon\R^7\mapsto\R\cup\{+\infty\}$ such that $\Phi(\nu)=\mathcal G(m(\nu))$ for all $\nu\in\R^3$.
\end{theorem*}

\begin{proof}[Second proof of Theorem~\ref{thm:main} (via signed singular values)]
We consider the function $\widetilde W_\psi$ from \eqref{eq:Wtilde}, which is defined by the same formula for all $F\in\R^{3\times3}$. Since $\Cof(Q_1FQ_2)=Q_1(\Cof F)Q_2$ for all $Q_1,Q_2\in\SO(3)$ and singular values are invariant under multiplication by rotations, $\widetilde W_\psi$ is isotropic; moreover, it is continuous, since $\bar g$ is convex and finite on $[0,\infty)$ and, hence, continuous on $(0,\infty)$, while $\nu_i(F+\Cof F)\geq0$ and $\bar g$ is non-decreasing, so that it is also continuous at $0$ from the right. For $\nu\in\R^3$, one has
\begin{equation}\label{eq:cof-diag}
  \Cof\operatorname{diag}(\nu_1,\nu_2,\nu_3)=\operatorname{diag}(\nu_2\nu_3,\nu_1\nu_3,\nu_1\nu_2),
\end{equation}
so that the singular values of $\operatorname{diag}(\nu)+\Cof\operatorname{diag}(\nu)$ are $|\nu_1+\nu_2\nu_3|$, $|\nu_2+\nu_1\nu_3|$ and $|\nu_3+\nu_1\nu_2|$. Hence, the function $\Phi$ associated with $\widetilde W_\psi$ by \eqref{eq:Phi-ssv} is
\begin{equation}\label{eq:Phi-Wtilde}
  \Phi(\nu)=\bar g\mleft(|\nu_1+\nu_2\nu_3|\mright)+\bar g\mleft(|\nu_2+\nu_1\nu_3|\mright)+\bar g\mleft(|\nu_3+\nu_1\nu_2|\mright)=\mathcal G(m(\nu)),
\end{equation}
where
\begin{equation}\label{eq:Q}
  \mathcal G\colon\R^3\times\R^3\times\R\mapsto\R,\qquad \mathcal G(x,y,d)=Q(x,y)\coloneq\sum_{i=1}^3\bar g\mleft(|x_i+y_i|\mright).
\end{equation}
Since $\bar g$ is convex and non-decreasing on $[0,\infty)$ and $r\mapsto|r|$ is convex on $\R$, the function $r\mapsto\bar g(|r|)$ is convex on $\R$. As $(x,y,d)\mapsto x_i+y_i$ is linear, $\mathcal G$ is convex and finite, hence continuous. By \cite[Theorem~1.4]{WiedemannPeter2026}, $\widetilde W_\psi$ is polyconvex on $\R^{3\times3}$ and, in particular, on $\GLp(3)$.

It remains to identify the restriction to $\SL(3)$. If $F\in\SL(3)$ and $\nu$ is a triple of signed singular values of $F$, then $\nu_1\nu_2\nu_3=1$. Thus $\nu_j\nu_k=\nu_i^{-1}$ for $\{i,j,k\}=\{1,2,3\}$, and $\nu_i$ and $\nu_i^{-1}$ have the same sign. Therefore, by isotropy,
\begin{equation}\label{eq:Q-minc}
  \widetilde W_\psi(F)=\Phi(\nu)=\sum_{i=1}^3\bar g\mleft(|\nu_i+\nu_i^{-1}|\mright)=\sum_{i=1}^3\bar g\mleft(\lambda_i+\lambda_i^{-1}\mright)=\sum_{i=1}^3\psi\mleft(|\!\log\lambda_i|\mright)=W_\psi(F),
\end{equation}
where, as in \eqref{eq:W-P}, we have used $\lambda_i+\lambda_i^{-1}\geq2$, $\bar g=g$ on $[2,\infty)$, and $\arcosh\frac{\lambda_i+\lambda_i^{-1}}{2}=|\!\log\lambda_i|$.
\end{proof}

On the incompressibility constraint, only the first six components of $m(\nu)$ are relevant, i.e., $m_{\rm inc}(\nu)=(\nu_1,\nu_2,\nu_3,\nu_2\nu_3,\nu_1\nu_3,\nu_1\nu_2)$, see \cite{Geuken2026}, and $Q$ is an explicit convex representative in this sense. Moreover, the matrix-level convex representative $P(X,Y,d)=G(X+Y)$ from \eqref{eq:P-def} reduces to $\mathcal G$ on diagonal matrices. The first proof has the advantage of producing the convex representative $P$ directly at the matrix level and of relying only on Ball's classical result, whereas the second proof reduces the verification to the convexity of the explicit scalar function $\mathcal G$ of seven variables. In both proofs, the essential ingredients are the identification of the structure $F+\Cof F$ on $\SL(3)$, which reduces the problem to the three sums $\nu_i+\nu_j\nu_k=\nu_i+\nu_i^{-1}$, and the resulting scalar condition on $g$, which can be verified by means of the differential criterion of Proposition~\ref{prop:differential}.}

On ${\rm SL}(3)$, the representation \eqref{eq:Wtilde} depends only on the reciprocal-invariant quantities \(\lambda_i+\lambda_i^{-1}\) and, therefore, naturally yields an inversion-symmetric representative of the scalar generator. As shown in Appendix~A, term-wise inversion symmetry of a prescribed generator is not necessary, since a logarithmic null term \(c\log\lambda\) is invisible on \(\SL(3)\).
The rank-one convexity conclusion also follows directly from the same convex
spectral representation.

\begin{proposition}[Direct proof on admissible rank-one lines]\label{prop:direct}
Under the assumptions of \cref{thm:main}, let
$F_t=F+t\,a\otimes b$ with $F\in\SL(3)$ and
 {$\langle b, F^{-1}a\rangle=0$}. Then $t\mapsto W_\psi(F_t)$ is convex \ {over the unit interval}.
\end{proposition}

\begin{proof}
By \eqref{eq:detlemma}, $\det F_t=1$ for every $t$. The
Sherman--Morrison formula \cite{ShermanMorrison1950} gives
\begin{equation}\label{eq:SM}
  F_t^{-1}
  =F^{-1}-t\,(F^{-1}a)\otimes(F^{-T}b).
\end{equation}
After transposition,
\begin{equation}\label{eq:cofFt}
  \Cof F_t=F_t^{-T}
  =F^{-T}-t\,(F^{-T}b)\otimes(F^{-1}a).
\end{equation}
Therefore,
\begin{align}\label{eq:affine-Fcof}
  F_t+\Cof F_t
  ={}&F+F^{-T}+t\Bigl[a\otimes b-(F^{-T}b)\otimes(F^{-1}a)\Bigr],
\end{align}
which is affine in $t$. Since $G$ is convex,
\begin{equation}\label{eq:G-line}
  t\mapsto G(F_t+\Cof F_t)
\end{equation}
is convex. As $F_t\in\SL(3)$, \cref{thm:main} gives
$G(F_t+\Cof F_t)=W_\psi(F_t)$, proving the assertion.
\end{proof}

\subsection{A differential sufficient condition}

\begin{proposition}\label{prop:differential}
Let $\psi\in C^2([0,\infty))$ satisfy
\begin{equation}\label{eq:origin}
  \psi^\prime(0)=0\qquad\text{and}\qquad\psi^\prime(x)\geq0\qquad\forall\,x>0,
\end{equation}
and
\begin{equation}\label{eq:diffcond}
  \psi^{\prime\prime}(x)\sinh x-\psi^\prime(x)\cosh x\geq0
\qquad\forall\,x>0,
\end{equation}
or equivalently,
\begin{equation}\label{eq:diffcond-coth}
  \psi^{\prime\prime}(x)\geq\psi^\prime(x)\coth x\qquad\forall\,x>0.
\end{equation}
Then, the function
\begin{equation}\label{eq:gdef}
	g(s)=\psi\mleft(\arcosh\frac{s}{2}\mright)\qquad\forall\,s\geq2.
\end{equation} is convex and non-decreasing on
$[2,\infty)$. Moreover,
\begin{equation}\label{eq:gbarext}
  \bar g(s)=
  \begin{cases}
    \psi(0)+\frac{1}{2}\psi^{\prime\prime}(0)(s-2),& \text{if }0\leq s\leq2,\\[1em]
    \psi\mleft(\arcosh\frac{s}{2}\mright),&\text{if }s\geq2,
  \end{cases}
\end{equation}
is a convex and non-decreasing extension of $g${, which, in contrast to the extension $s\mapsto g(\max\{s,2\})$, is of class $C^1$ at $s=2$}. Consequently,
{Theorem~\ref{thm:main}} applies.
\end{proposition}
\begin{proof}
For $s>2$, set
$
  s=2\cosh x,$ $ x=\arcosh\frac{s}{2}>0.
$
Since $ds/dx=2\sinh x$, we have
$
  g^\prime(s)=\frac{\psi^\prime(x)}{2\sinh x}
$
and
\begin{equation}\label{eq:gsecond}
  g^{\prime\prime}(s)
  =\frac{\psi^{\prime\prime}(x)\sinh x-\psi^\prime(x)\cosh x}
         {4\sinh^3x}.
\end{equation}
Thus,  $g$ is non-decreasing
and convex on $(2,\infty)$.
Because $\psi^\prime(0)=0$ and $\psi\in C^2([0,\infty))$,
\begin{equation}\label{eq:gprime-limit}
  \lim_{s\to2}g^\prime(s)
  =\lim_{x\to0}\frac{\psi^\prime(x)}{2\sinh x}
  =\frac{\psi^{\prime\prime}(0)}{2}.
\end{equation}
Furthermore,  we have that 
$\psi^{\prime\prime}(0)\geq0$; indeed, $\psi^\prime$ has a minimum at the left endpoint $0$.
Hence the affine branch in \eqref{eq:gbarext} is non-decreasing. It agrees with
$g$ in both value and right derivative at $s=2$. Since $g^\prime$ is non-decreasing
on $(2,\infty)$, the derivative of the piecewise function \eqref{eq:gbarext}
is non-decreasing across the junction. Therefore $\bar g$ is convex and
non-decreasing on $[0,\infty)$.
\end{proof}
In the equality case, $\psi^{\prime\prime}(x)=\psi^\prime(x)\coth x$ in \eqref{eq:diffcond-coth}, one has
$(\log|\psi^\prime(x)|)^\prime=\coth x$, wherever $\psi^\prime(x)\neq0$. Hence
$\psi^\prime(x)=C_1\sinh x$ and, after a second integration,
$
  \psi(x)=C_1\cosh x+C_2,
$
where the condition $\psi^\prime(0)=0$ is automatically satisfied.
\begin{remark}[Ellipticity and Hill monotonicity]
\label{rem:ellipticity-hill}
The conditions of Proposition \ref{prop:differential} provide a one-dimensional
differential sufficient condition for rank-one convexity and, hence, for Legendre--Hadamard ellipticity along admissible rank-one directions
in $\SL(3)$. Moreover, since
$
  \psi^{\prime\prime}(x)\geq \psi^\prime(x)\coth x \geq 0
$
for all $x > 0$, the function $\psi$ is convex on $[0,\infty)$. Together with
$\psi^\prime(0)=0$, this implies that its even extension
$
  \vartheta\colon\R\mapsto\R,
  \ 
  \vartheta(x)\coloneq\psi(|x|),
$
is convex on $\R$. Indeed, for $x\neq0$,
$
  \vartheta^\prime(x)
  =
  \psi^\prime(|x|)\frac{d}{dx}|x|
  =
  \sgn(x)\psi^\prime(|x|),
$
while $\psi^\prime(0)=0$ yields $\vartheta^\prime(0)=0$. Thus, with the convention
$\sgn(0)=0$, one may write
\begin{equation}\label{eq:vartheta-prime}
  \vartheta^\prime(x)=\sgn(x)\psi^\prime(|x|)
  \qquad\forall\,x\in\R.
\end{equation}
The sign factor is therefore essential: it is precisely the derivative of
the absolute value entering the even extension $\vartheta(x)=\psi(|x|)$.
Consider now the logarithmic principal stretches
$
  x_i=\log\lambda_i .
$
The incompressibility constraint $\lambda_1\lambda_2\lambda_3=1$ is
equivalent to
$
  x_1+x_2+x_3=0.
$
Hence the corresponding reduced logarithmic energy is
\begin{equation}\label{eq:log-reduced}
  \widehat W(x_1,x_2,x_3)
  =\sum_{i=1}^3\vartheta(x_i)\qquad\forall\,\sum_{i=1}^3 x_i = 0.
\end{equation}
Equivalently, in terms of two independent logarithmic variables,
\begin{equation}\label{eq:log-reduced-two-variables}
  \widehat W_{\mathrm{red}}^{\mathrm{inc}}(x_1,x_2)
  \coloneq\widehat W(x_1,x_2,-x_1-x_2)
  =\vartheta(x_1)+\vartheta(x_2)+\vartheta(-x_1-x_2).
\end{equation}
Since $\vartheta$ is convex, the function
$(x_1,x_2,x_3)\mapsto\sum_i \vartheta(x_i)$ is convex on $\R^3$, and
therefore its restriction to the trace-free plane
$
  \mathcal H
  \coloneq
  \mleft\{x\in\R^3\,|\,\sum_i x_i=0\mright\}
$
is convex as well.

For the unconstrained logarithmic variables, the principal components of
the corresponding Kirchhoff stress are given by
\[
  \tau_i
  =
  \frac{\partial\widehat W}{\partial x_i}
  =
  \vartheta^\prime(x_i)
  =
  \sgn(x_i)\psi^\prime(|x_i|).
\]

In the incompressible setting the Kirchhoff stress is determined only up to
an arbitrary hydrostatic contribution, i.e., up to addition of the same
scalar to all three principal components. This ambiguity is immaterial on
$\mathcal H$, since for $x,y\in\mathcal H$ one has
$\sum_i (x_i-y_i)=0$. Thus, the monotonicity of the gradient of the convex
energy restricted to $\mathcal H$ can be written directly as
\begin{equation}\label{eq:hill-principal}
  \sum_{i=1}^3
  \bigl[
    \sgn(x_i)\psi^\prime(|x_i|)
    -
    \sgn(y_i)\psi^\prime(|y_i|)
  \bigr](x_i-y_i)
  \geq0,
\end{equation}
for all $x,y\in\R^3$ satisfying
$\sum_i x_i=\sum_i y_i=0$.

In other words, \eqref{eq:hill-principal} is simply the standard monotonicity
inequality
\[
  \bigl\langle{\rm D}\widehat W(x)-{\rm D}\widehat W(y),x-y\bigr\rangle
  \geq0
\]
restricted to the trace-free logarithmic-strain plane, and corresponds precisely to the convexity of $  \widehat W_{\mathrm{red}}^{\mathrm{inc}}$. It is the weak form of Hill's
inequality in the principal logarithmic strains.
For related logarithmic stress--strain considerations and for the
corresponding implication in the incompressible planar setting, see
\cite{NBecker,GhibaNeffWollner}.
\end{remark}

For $\psi(x)=h(x^2)$, \eqref{eq:diffcond} takes the following form.

\begin{corollary}\label{cor:h}
Let $h\in C^2([0,\infty))$ satisfy $h'(r)\geq0$ for all $r\geq0$ and
\begin{equation}\label{eq:hcond}
  \bigl[h^\prime(x^2)+2x^2h^{\prime\prime}(x^2)\bigr]\tanh x
  \geq x\,h^\prime(x^2),
  \qquad x>0.
\end{equation}
Then the function on $\SL(3)$ given by
\begin{equation}\label{eq:h-energy}
  F\mapsto\sum_{i=1}^3h\mleft(\log^2\!\lambda_i\mright)
\end{equation}
admits the polyconvex extension supplied by \cref{thm:main}.
\end{corollary}
\begin{proof}
For $\psi(x)=h(x^2)$,
$
  \psi^\prime(x)=2x\,h^\prime(x^2),$ $
  \psi^{\prime\prime}(x)=2h^\prime(x^2)+4x^2h^{\prime\prime}(x^2).
$
The condition $\psi^{\prime\prime}(x)\tanh x\geq\psi^\prime(x)$ is exactly
\eqref{eq:hcond}, while $\psi^\prime(0)=0$ and $\psi^\prime\geq0$ follow from
$h^\prime\geq0$. Apply {Theorem~\ref{thm:main} and Proposition~\ref{prop:differential}}.
\end{proof}
\section{An example and a counterexample}
\subsection{The exponential inversion-symmetric Valanis--Landel type energy}\setcounter{equation}{0}
\begin{proposition}\label{thm:exp}
Let $\mu>0$ and $k\geq\frac16$. The energy
\begin{equation}\label{eq:Wexp}
  W_{\exp,k}(F)=\frac{\mu}{k}\sum_{i=1}^3
  e^{k\log^2\!\lambda_i}\qquad\forall\,F\in\SL(3),
\end{equation}
is the restriction to $\SL(3)$ of the polyconvex function
\ {\begin{equation}\label{eq:Wexptilde}
  {\widetilde W_{\exp,k}\colon\GLp(3)\mapsto\R,\qquad\widetilde W_{\exp,k}(F)=P_k(F,\Cof F,\det F)
  =\sum_{i=1}^3\bar g_k\mleft(\nu_i(F+\Cof F)\mright),}
\end{equation}}
{with the convex function $P_k(X,Y,d)=\sum_{i=1}^3\bar g_k\mleft(\nu_i(X+Y)\mright)$,} where
\begin{equation}\label{eq:gexpbar}
  \bar g_k(s)=
  \begin{cases}
    \displaystyle \frac{\mu}{k}+\mu(s-2),
      &\text{if }0\leq s\leq2,\\[1em]
    \displaystyle \frac{\mu}{k}
    \exp\mleft(k\,\arcosh^2\!\mleft(\frac{s}{2}\mright)\mright),
      &\text{if }s\geq2.
  \end{cases}
\end{equation}
Consequently, $W_{\exp,k}$ is rank-one convex on $\SL(3)$.
\end{proposition}

\begin{proof}
Set $\psi_k(x)=\frac{\mu}{k}e^{kx^2}$ for $x\geq0$. Then
$\psi_k^\prime(x)=2\mu x e^{kx^2}\geq0$,
$\psi_k^{\prime\prime}(x)=2\mu(1+2kx^2)e^{kx^2}$,
$\psi_k^\prime(0)=0$, and $\psi_k^{\prime\prime}(0)=2\mu$. Moreover,
\begin{equation}\label{eq:exp-k-diff}
 \psi_k^{\prime\prime}(x)\sinh x-\psi_k^\prime(x)\cosh x
 =2\mu e^{kx^2}
 \bigl[(1+2kx^2)\sinh x-x\cosh x\bigr].
\end{equation}
Thus it remains to show that
\[
 N_k(x)\coloneq(1+2kx^2)\sinh x-x\cosh x\geq0\qquad\forall\,x>0.
\]
For $k\geq\frac16$, we write
\[
 N_k(x)=
 \mleft(\mleft(1+\frac{x^2}{3}\mright)\sinh x-x\cosh x\mright)
 +2\mleft(k-\frac16\mright)x^2\sinh x .
\]
Let
$M(x)=\mleft(1+\frac{x^2}{3}\mright)\sinh x-x\cosh x$.
Then {$M(0)=0$} and
$M^\prime(x)=\frac{x}{3}\mleft(x\cosh x-\sinh x\mright)$. Since
$x\cosh x-\sinh x$ vanishes at $x=0$ and has derivative $x\sinh x>0$
for $x>0$, we obtain $M^\prime(x)>0$ and, hence, $M(x)>0$ for every $x>0$.
Consequently, $N_k(x)>0$ for $x>0$, whenever $k\geq\frac16$.

The hypotheses of {Proposition~\ref{prop:differential}} are thus satisfied. In this case
\[
 g_k(s)=\frac{\mu}{k}
 \exp\mleft(k\,\arcosh^2\!\mleft(\frac{s}{2}\mright)\mright)\qquad\forall\,s\geq2.
\]
Since $\psi_k(0)=\frac{\mu}{k}$ and $\tfrac{1}{2}\psi_k^{\prime\prime}(0)=\mu$, formula
\eqref{eq:gbarext} gives
$\bar g_k(s)=\frac{\mu}{k}+\mu(s-2)$ on $[0,2]$, while on $[2,\infty)$
it coincides with $g_k$. Hence \eqref{eq:gexpbar} is a convex and
non-decreasing extension of $g_k$, and the conclusion follows from
{Theorem~\ref{thm:main}}.
\end{proof}
{\begin{remark*}[Sharpness of the threshold $k\geq\frac16$]
The threshold $k\geq\frac16$ in Proposition~\ref{thm:exp} is sharp with respect to condition \eqref{eq:diffcond}. Indeed, using the Taylor expansions $\sinh x=x+\frac{x^3}{6}+O(x^5)$ and $\cosh x=1+\frac{x^2}{2}+O(x^4)$, we obtain
\begin{equation}\label{eq:Nk-expansion}
  N_k(x)=(1+2kx^2)\sinh x-x\cosh x=\mleft(2k-\frac13\mright)x^3+O(x^5)\qquad\text{as }x\to0.
\end{equation}
Hence, for $0<k<\frac16$, one has $N_k(x)<0$ for all sufficiently small $x>0$, and condition \eqref{eq:diffcond} fails near $x=0$. In view of \eqref{eq:gsecond} and \eqref{eq:exp-k-diff}, this also gives $g_k^{\prime\prime}(s)<0$ for $s>2$ sufficiently close to $2$, so that $g_k$ is not convex on $[2,\infty)$ and the hypothesis of Theorem~\ref{thm:main} is not satisfied. This shows that the construction of Theorem~\ref{thm:main} does not apply to $W_{\exp,k}$ for $0<k<\frac16$; it is not asserted that $W_{\exp,k}$ fails to be rank-one convex on $\SL(3)$ for these values of $k$.
\end{remark*}}
\begin{figure}[ht]
\centering

\begin{subfigure}[t]{0.48\textwidth}
\centering
\begin{tikzpicture}
\begin{axis}[
  width=\textwidth,height=7.2cm,
  domain=0.1:10,samples=200,
  axis lines=middle,
  xmin=0.1,xmax=10,ymin=0,ymax=6,
  xlabel={$\lambda$},ylabel={scalar energy},
  tick label style={black},
  label style={black},
  axis line style={black}]
\addplot[black,thick] {(ln(x))^2};
\end{axis}
\end{tikzpicture}
\caption{The quadratic energy $\lambda\mapsto\log^2\!\lambda$.}
\label{fig:quadratic-energy}
\end{subfigure}
\hfill
\begin{subfigure}[t]{0.48\textwidth}
\centering
\begin{tikzpicture}
\begin{axis}[
  width=\textwidth,height=7.2cm,
  domain=0.1:10,samples=200,
  axis lines=middle,
  xmin=0.1,xmax=10,ymin=0,ymax=50,
  xlabel={$\lambda$},ylabel={scalar energy},
  tick label style={black},
  label style={black},
  axis line style={black}]
\addplot[black,thick] {exp((ln(x))^2)-1};
\end{axis}
\end{tikzpicture}
\caption{The exponential energy $\lambda\mapsto e^{\log^2\!\lambda}-1${, i.e., the scalar summand of \eqref{eq:Wexp} for $\mu=k=1$, shifted by the additive constant $-1$ so that it vanishes at $\lambda=1$, as the quadratic energy does}.}
\label{fig:exponential-energy}
\end{subfigure}

\caption{Comparison of two inversion-symmetric scalar energies.}
\label{fig:scalar-energies}
\end{figure}
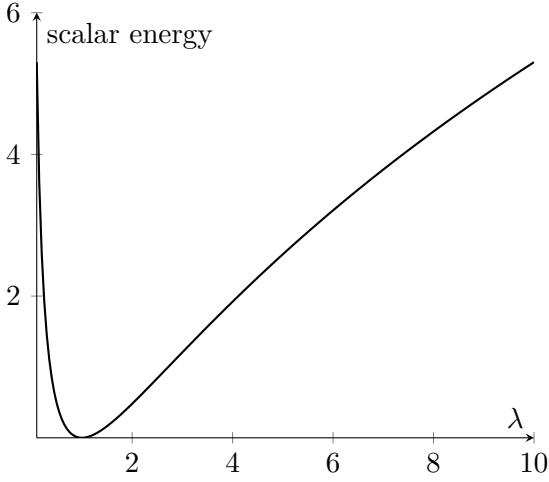
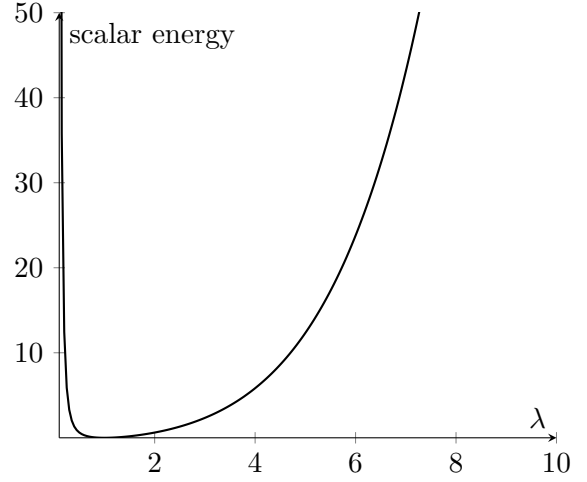

On \(\SL(3)\), up to an irrelevant additive constant, the exponentiated Hencky energy is
$ W_{\rm eH}(F)=\frac{\mu}{k} \exp\mleft(k\|\log V\|^2\mright) =\frac{\mu}{k} \exp\mleft(k\sum_i \log^2\lambda_i\mright), $
and is therefore different from the term-wise exponential energy considered here, i.e.,
$ W_{\exp,k}(F)=\frac{\mu}{k}\sum_i \exp\mleft(k\log^2\lambda_i\mright). $
The three-dimensional exponentiated Hencky energy is not globally rank-one convex \cite{NeffGhibaLankeit}.
In this sense, the energy considered here is another interesting generalisation of the quadratic Hencky energy. 

\subsection{Scalar convexity in \texorpdfstring{$\log\lambda$}{log lambda} alone is not sufficient}

\begin{proposition}\label{prop:hencky}
The quadratic Hencky energy \cite{agn_neff2015geometry}
\begin{equation}\label{eq:WH}
  W_H(F)=\mu\,\sum_{i=1}^3\log^2\!\lambda_i\qquad\forall\,F\in\SL(3),
\end{equation}
is not rank-one convex on $\SL(3)$.
\end{proposition}

\begin{proof}
Consider the simple-shear line
\begin{equation}\label{eq:shear-line}
  F_t=\id_3+t\,e_1\otimes e_2
  =\setlength\arraycolsep{0.5em}\begin{pmatrix}1&t&0\\0&1&0\\0&0&1\end{pmatrix}.
\end{equation}
It is a rank-one line contained in $\SL(3)$. For $t\geq0$, its singular values
are
$
  e^{\alpha(t)}, $ $e^{-\alpha(t)},$ and $1$ with
$
  \alpha(t)=\arsinh\frac{t}{2}.
$
Indeed, the two planar singular values have product one and difference $t$.
Therefore,
\begin{equation}\label{eq:WH-shear}
  W_H(F_t)=2\mu\arsinh^2\!\mleft(\frac{t}{2}\mright).
\end{equation}
For $t>0$,
\begin{equation}\label{eq:henckysecond}
  \frac{d^2}{dt^2}W_H(F_t)
  =\mu\,\frac{4}{t^2+4}
  \mleft(1-\frac{t}{\sqrt{t^2+4}}\arsinh\mleft(\frac{t}{2}\mright)\mright).
\end{equation}
Choose $t=2\sinh2$. Then
$
  \arsinh\frac{t}{2}=2,
$ $
  \frac{t}{\sqrt{t^2+4}}=\tanh2,
$
and the bracket in \eqref{eq:henckysecond} equals
$1-2\tanh2<0$. Thus the restriction of $W_H$ to this rank-one line is not
convex.
\end{proof}

The preceding proposition also shows that convexity and monotonicity of
$h\colon[0,\infty)\mapsto\R$ do not suffice for the rank-one convexity of
$\sum_i h(\log^2\!\lambda_i)$; the choice $h(r)=r$ is convex and
non-decreasing, but gives $W_H$. 

\begin{remark}[Scope of the criterion]
The condition in \cref{thm:main} is a sufficient condition for the explicit
extension \eqref{eq:Wtilde}. It is not asserted to be necessary for rank-one
convexity on $\SL(3)$. The quadratic Hencky energy in {Proposition~\ref{prop:hencky}} also provides an
example showing that tension--compression (inversion) symmetry in the stretch
variable $\lambda$ alone is not sufficient for polyconvexity; the energy is
inversion-symmetric, but it is not even rank-one convex on $\SL(3)$.
\end{remark}

\section{Conclusion}

On $\SL(3)$, the identity
$
  F+\Cof F=R\,(U+U^{-1})
$
encodes the inversion symmetry of the principal stretches through
$
  \nu_i(F+\Cof F)=2\cosh(|\!\log\lambda_i|),
$
up to permutation. Consequently, every energy of the form
$\sum_i \psi(|\!\log\lambda_i|)$ for which
$s\mapsto\psi(\arcosh(\tfrac{s}{2}))$ {is convex and non-decreasing on $[2,\infty)$} admits an explicit polyconvex extension on $\GLp(3)$. The
criterion applies, e.g., to $\psi(x)=\frac{\mu}{k}e^{k\,x^2}$ and yields the rank-one convexity on
$\SL(3)$ of $\sum_i e^{\log^2\lambda_i}$. The simple-shear calculation for
the quadratic Hencky energy confirms that ordinary scalar convexity
 {in the logarithmic stretch variable} does not suffice. {We emphasize that the criterion is only sufficient: it provides a particular explicit polyconvex extension, but it does not characterize rank-one convexity or polyconvexity of $W_\psi$ on $\SL(3)$. For instance, for the exponential family the threshold $k\geq\frac16$ is sharp for the applicability of the criterion, while the rank-one convexity of $W_{\exp,k}$ on $\SL(3)$ for $0<k<\frac16$ remains open.}

{\bf Acknowledgements} The work of I.D. Ghiba  was supported by a grant of the Ministry of Education and Research, CNCS--UEFISCDI,
		project number PN-IV-P1-PCE2026-0163, within PNCDI IV and it benefits on the logistics of the project RISE-UAIC, code CNFIS-FDI-2026-F-0549.
\begin{footnotesize}
	\bibliographystyle{plain} 

\addcontentsline{toc}{section}{References}
 
\begin{thebibliography}{10}

\bibitem{Antman95}
S.~Antman.
\newblock {\em Nonlinear {P}roblems of {E}lasticity.}, volume 107 of {\em
  Applied {M}athematical {S}ciences}.
\newblock Springer, Berlin, 1995.

\bibitem{aubert1995necessary}
G.~Aubert.
\newblock Necessary and sufficient conditions for isotropic rank-one convex
  functions in dimension 2.
\newblock {\em Journal of Elasticity}, 39(1):31--46, 1995.

\bibitem{Ball1976}
J.~M. Ball.
\newblock Convexity conditions and existence theorems in nonlinear elasticity.
\newblock {\em Archive for Rational Mechanics and Analysis}, 63(4):337--403,
  1976.

\bibitem{GhibaNeffWollner}
I.D. Ghiba, M.P. Wollner, and P.~Neff.
\newblock Polyconvexity implies {Hill}'s inequality in {${\rm SL}(2)$}.
\newblock {\em European Journal of Mechanics-A/Solids}, 121:106296, 2026.

\bibitem{Geuken2026}
{G.-L. Geuken, P.~Kurzeja, D.~Wiedemann, M.~Zlati\'{c}, M.~\v{C}ana\dj ija, and J.~Mosler.
\newblock Modeling isotropic polyconvex hyperelasticity by neural networks---sufficient and necessary criteria for compressible and incompressible materials.
\newblock {\em Journal of the Mechanics and Physics of Solids}, 216:106768, 2026.}

\bibitem{Kuczma2009}
M.~Kuczma.
\newblock {\em An Introduction to the Theory of Functional Equations and
  Inequalities: Cauchy's Equation and Jensen's Inequality}.
\newblock Birkh{\"a}user, Basel, 2 edition, 2009.

\bibitem{Landel1998}
R.~F. Landel.
\newblock A simple {$w'(\lambda)$} function for the {Valanis--Landel} form of
  stored energy function.
\newblock {\em Rubber Chemistry and Technology}, 71(2):234--243, 1998.

\bibitem{MartinEtAlBiot2025}
R.~J. Martin, I.-D. Ghiba, M.~K{\"o}hler, D.~Balzani, O.~Sander, and P.~Neff.
\newblock Quasiconvex relaxation of planar {Biot}-type energies and the role of
  determinant constraints.
\newblock {\em arXiv preprint arXiv:2501.10853}, 2025.

\bibitem{MartinGhibaNeff2019}
R.~J. Martin, I.-D. Ghiba, and P.~Neff.
\newblock A polyconvex extension of the logarithmic {Hencky} strain energy.
\newblock {\em Analysis and Applications}, 17(3):349--361, 2019.

\bibitem{mooney1940theory}
M.~Mooney.
\newblock A theory of large elastic deformation.
\newblock {\em Journal of Applied Physics}, 11(9):582--592, 1940.

\bibitem{agn_neff2015geometry}
P.~Neff, B.~Eidel, and R.~J. Martin.
\newblock Geometry of logarithmic strain measures in solid mechanics.
\newblock {\em Archive for Rational Mechanics and Analysis}, 222:507--572,
  2016.

\bibitem{NeffGhibaLankeit}
P.~Neff, I.~D. Ghiba, and J.~Lankeit.
\newblock The exponentiated {H}encky-logarithmic strain energy. {P}art {I}:
  {C}onstitutive issues and rank--one convexity.
\newblock {\em Journal of Elasticity}, 121:143--234, 2015.

\bibitem{agn_neff2014grioli}
P.~Neff, J.~Lankeit, and A.~Madeo.
\newblock On {Grioli}'s minimum property and its relation to {Cauchy}'s polar
  decomposition.
\newblock {\em International Journal of Engineering Science}, 80:209--217,
  2014.

\bibitem{NBecker}
P.~Neff, I.~M{\"u}nch, and R.J. Martin.
\newblock Rediscovering {G.~F. Becker}'s early axiomatic deduction of a
  multiaxial nonlinear stress--strain relation based on logarithmic strain.
\newblock {\em Mathematics and Mechanics of Solids}, 21(7):856--911, 2016.

\bibitem{Ogden83}
R.~W. Ogden.
\newblock {\em Non-{L}inear {E}lastic {D}eformations.}
\newblock Mathematics and its Applications. Ellis Horwood, Chichester, 1983.

\bibitem{PengLandel1972}
T.~J. Peng and R.~F. Landel.
\newblock Stored energy function of rubberlike materials derived from simple
  tensile data.
\newblock {\em Journal of Applied Physics}, 43(7):3064--3067, 1972.

\bibitem{rivlin1948large}
R.~Rivlin.
\newblock Large elastic deformations of isotropic materials. {IV.} { F}urther
  developments of the general theory.
\newblock {\em Philosophical Transactions of the Royal Society of London.
  Series A, Mathematical and Physical Sciences}, 241(835):379--397, 1948.

\bibitem{Rivlin2003}
R.~S. Rivlin.
\newblock The {Valanis--Landel} strain-energy function.
\newblock {\em Journal of Elasticity}, 73(1--3):291--297, 2003.

\bibitem{rivlin2006relation}
R.~S. Rivlin.
\newblock The relation between the {Valanis--Landel} and classical
  strain-energy functions.
\newblock {\em International Journal of Non-Linear Mechanics}, 41(1):141--145,
  2006.

\bibitem{ShermanMorrison1950}
J.~Sherman and W.~J. Morrison.
\newblock Adjustment of an inverse matrix corresponding to a change in one
  element of a given matrix.
\newblock {\em The Annals of Mathematical Statistics}, 21(1):124--127, 1950.

\bibitem{Shield1967}
R.~T. Shield.
\newblock Inverse deformation results in finite elasticity.
\newblock {\em Zeitschrift f{\"u}r Angewandte Mathematik und Physik},
  18(4):490--500, 1967.

\bibitem{SussmanBathe2009}
T.~Sussman and K.-J. Bathe.
\newblock A model of incompressible isotropic hyperelastic material behavior
  using spline interpolations of tension--compression test data.
\newblock {\em Communications in Numerical Methods in Engineering},
  25(1):53--63, 2009.

\bibitem{Valanis2022}
K.~C. Valanis.
\newblock The {Valanis--Landel} strain energy function: Elasticity of
  incompressible and compressible rubber-like materials.
\newblock {\em International Journal of Solids and Structures}, 238:111271,
  2022.

\bibitem{valanis1967strain}
K.C. Valanis and R.F. Landel.
\newblock The strain-energy function of a hyperelastic material in terms of the
  extension ratios.
\newblock {\em J. Appl. Phys.}, 38(7):2997--3002, 1967.

\bibitem{VitucciTrentadueDeTommasi2026}
G.~Vitucci, F.~Trentadue, and D.~De Tommasi.
\newblock Inversion symmetry in separable line, area and volume elastic
  energies.
\newblock {\em Continuum Mechanics and Thermodynamics}, 38:40, 2026.

\bibitem{WiedemannPeter2026}
{D.~Wiedemann and M.~A. Peter.
\newblock Characterization of polyconvex isotropic functions.
\newblock {\em Calculus of Variations and Partial Differential Equations}, 65:115, 2026.}

\bibitem{ZubovRudev}
L.M. {Zubov} and A.N. {Rudev}.
\newblock {A criterion for the strong ellipticity of the equilibrium equations
  of an isotropic nonlinearly elastic material.}
\newblock {\em {Journal of Applied Mathematics and Mechanics}}, 75:432--446,
  2011.

\end{thebibliography}

\appendix 

\section{Inversion symmetry and Shield invariance on
\texorpdfstring{$\mathrm{SL}(n)$}{SL(n)}}\setcounter{equation}{0}
\label{sec:inverse-symmetry-shield}

The notion of inversion symmetry is closely related to the Shield
transformation. However, one must distinguish between inversion symmetry of
the total stored energy and inversion symmetry of a scalar function appearing
in a separable Valanis--Landel representation. In this section, we clarify
this distinction and compare the cases $\mathrm{SL}(2)$ and
$\mathrm{SL}(3)$.

Recall that the Shield transform \cite{Shield1967} of a stored energy
$
W\colon\mathrm{GL}^+(n)\mapsto\mathbb{R}
$
is defined by
\begin{align}
  W^\#(F)\coloneq\det F\cdot W(F^{-1}).
  \label{eq:shield-transform-SLn}
\end{align}
The energy is called Shield-invariant if
\begin{align}
  W^\#(F)=W(F)
  \qquad\forall\,F\in\mathrm{GL}^+(n).
  \label{eq:shield-invariance-GL}
\end{align}

\subsection{Restriction of the Shield transformation to
\texorpdfstring{$\mathrm{SL}(n)$}{SL(n)}}

On the special linear group, the determinant factor in the Shield
transformation disappears.
Let
$
W\colon\mathrm{SL}(n)\mapsto\mathbb{R}.
$
Then the restriction of the Shield transform to $\mathrm{SL}(n)$ is given
by
\begin{align}
  W^\#(F)=W(F^{-1}).
\end{align}
Consequently,
\begin{align}
  W \text{ is Shield-invariant on }\mathrm{SL}(n)
  \qquad\Longleftrightarrow\qquad
  W(F)=W(F^{-1})\qquad\forall\,F\in\mathrm{SL}(n).
  \label{eq:shield-inverse-equivalence}
\end{align}
Thus, on $\mathrm{SL}(n)$, Shield invariance is exactly the {inversion}
symmetry of the total stored energy.

This condition must be distinguished from the point-wise scalar condition
\begin{align}
  w(\lambda)=w(\lambda^{-1})\qquad\forall\,
  \lambda>0,
  \label{eq:inverse-symmetry-generator}
\end{align}
which may be imposed on a generating function in a Valanis--Landel
representation.

\subsection{The two-dimensional case}

The situation on $\mathrm{SL}(2)$ is particularly simple and transparent.

\begin{proposition}
\label{prop:automatic-shield-SL2}
Every objective and isotropic energy
$
W\colon\mathrm{SL}(2)\mapsto\mathbb{R}
$
is inversion-symmetric and therefore Shield-invariant.
\end{proposition}

\begin{proof}
Let $\lambda_1,\lambda_2>0$ denote the singular values of
$F\in\mathrm{SL}(2)$. Since
$
\lambda_1\lambda_2=1,
$
there exists $\lambda>0$ such that
$
(\lambda_1,\lambda_2)
=
(\lambda,\lambda^{-1}).
$
By objectivity and isotropy, the energy can be written as
$
W(F)=g(\lambda_1,\lambda_2),
$
where $g$ is symmetric:
$
g(\lambda_1,\lambda_2)
=
g(\lambda_2,\lambda_1).
$
The singular values of $F^{-1}$ are
$
(\lambda_1^{-1},\lambda_2^{-1})
=
(\lambda^{-1},\lambda).
$
Therefore,
$
W(F^{-1})
=
g(\lambda^{-1},\lambda)
=
g(\lambda,\lambda^{-1})
=
W(F),
$
and the conclusion follows.
\end{proof}

In particular, consider a separable Valanis--Landel energy
\begin{align}
  W(F)=w(\lambda_1)+w(\lambda_2).
  \label{eq:VL-SL2}
\end{align}
On $\mathrm{SL}(2)$, one has
\[
W(F)
=
w(\lambda)+w(\lambda^{-1}).
\]
Hence,
\[
W(F^{-1})
=
w(\lambda^{-1})+w(\lambda)
=
W(F)
\]
for every scalar function $w$.

\begin{remark}
\label{rem:SL2-generator-symmetry}
On $\mathrm{SL}(2)$, inversion symmetry of the total objective and isotropic
energy is automatic.  By contrast, the scalar condition
$
w(\lambda)=w(\lambda^{-1})
$
is an additional constitutive assumption and is not required for Shield
invariance of the total Valanis--Landel energy.
\end{remark}

\subsection{The three-dimensional case}

The corresponding property is not automatic on $\mathrm{SL}(3)$. If the
singular values of $F\in\mathrm{SL}(3)$ are
$
\lambda_1,\lambda_2,\lambda_3>0,
$ such that $
\lambda_1\lambda_2\lambda_3=1,
$
then the singular values of $F^{-1}$ are
$
\lambda_1^{-1},\lambda_2^{-1},\lambda_3^{-1}.
$
In general, the reciprocal triple is not a permutation of the original
triple. For instance,
$
\mleft(2,2,\frac14\mright)
\ \mapsto\ 
\mleft(\frac12,\frac12,4\mright).
$
Consequently, objectivity and isotropy alone do not imply Shield invariance
on $\mathrm{SL}(3)$.
For example, consider the objective and isotropic energy
$
W(F)=\lambda_1+\lambda_2+\lambda_3.
$
For
$
F=\operatorname{diag}\mleft(2,2,\frac14\mright),
$
one has
$
W(F)=2+2+\frac14=\frac{17}{4},
$
whereas
$
W(F^{-1})
=
\frac12+\frac12+4
=
5.
$
Thus,
$
W(F)\neq W(F^{-1}).
$

\subsection{Generalized Valanis--Landel energies on
\texorpdfstring{$\mathrm{SL}(3)$}{SL(3)}}

Consider a generalized Valanis--Landel energy of the form
\begin{align}
  W_{\mathrm{VL}}(F)
  =
  \sum_{i=1}^3 f_{\lambda}(\lambda_i)
  +
  \sum_{i=1}^3 f_{\omega}(\omega_i)
  +
  f_J(J),
  \label{eq:generalized-VL}
\end{align}
where
$
J=\lambda_1\lambda_2\lambda_3
$
is the local volume ratio and
$
\omega_1=\lambda_2\lambda_3, 
$ $
\omega_2=\lambda_3\lambda_1, 
$
 $
\omega_3=\lambda_1\lambda_2
$
are the principal area stretches.

On $\mathrm{SL}(3)$, one has
$
J=1$ and $\omega_i=\lambda_i^{-1}.
$
Therefore,
\begin{align}
  W_{\mathrm{VL}}(F)
  &=
  \sum_{i=1}^3 f_{\lambda}(\lambda_i)
  +
  \sum_{i=1}^3 f_{\omega}(\lambda_i^{-1})
  +
  f_J(1)
 =
  \sum_{i=1}^3 w(\lambda_i)+f_J(1),
  \label{eq:effective-VL-SL3}
\end{align}
where
$
  w(\lambda)
  \coloneq
  f_{\lambda}(\lambda)
  +
  f_{\omega}(\lambda^{-1}).
$
Thus, after restriction to $\mathrm{SL}(3)$, the separate line and area
contributions enter only through the effective scalar generator $w$.

\begin{proposition}
\label{prop:scalar-inverse-implies-shield}
Let
$
W(F)=\sum_i w(\lambda_i),
\ 
F\in\mathrm{SL}(3).
$
If
$
  w(\lambda)=w(\lambda^{-1})
  \ 
  \text{for all }\lambda>0,
$
then $W$ is Shield-invariant on $\mathrm{SL}(3)$.
\end{proposition}

\begin{proof}
The singular values of $F^{-1}$ are
$\lambda_1^{-1},\lambda_2^{-1},\lambda_3^{-1}$. Therefore,
$
W(F^{-1})
=
\sum_i w(\lambda_i^{-1})
=
\sum_i w(\lambda_i)
=
W(F)
$
and the conclusion follows.
\end{proof}

In particular, suppose that
$
  f_{\lambda}(\lambda)
  =
  f_{\lambda}(\lambda^{-1}),
\ 
  f_{\omega}(\lambda)
  =
  f_{\omega}(\lambda^{-1}).
  \label{eq:inverse-symmetry-area}
$
Then
$
w(\lambda^{-1})
=
f_{\lambda}(\lambda^{-1})
+
f_{\omega}(\lambda)
 =
f_{\lambda}(\lambda)
+
f_{\omega}(\lambda^{-1})
$ $ =
w(\lambda).
$
Hence, the {inversion-symmetry} assumptions imposed on the line and area
contributions provide a direct construction of Shield-invariant energies
on $\mathrm{SL}(3)$.

The converse must be formulated at the level of the total energy. Shield
invariance does not necessarily imply that a prescribed generating
function $w$ is pointwise inversion-symmetric{.}

\begin{theorem}
\label{thm:VL-shield-characterization-SL3}
Let $w\colon(0,\infty)\mapsto\mathbb{R}$ be continuous, and define
$
  W(F)=\sum_i w(\lambda_i)$ for all $F\in\mathrm{SL}(3)$.
Then the following statements are equivalent:

\begin{enumerate}
\renewcommand{\labelenumi}{(\roman{enumi})}

\item
The energy $W$ is Shield-invariant on $\mathrm{SL}(3)$.

\item
There exists a constant $c\in\mathbb{R}$ such that
$
  w(\lambda)-w(\lambda^{-1})
  =
  c\log\lambda
  \ 
  \text{for all }\lambda>0.
$

\item
There exists a continuous function
$\widetilde{w}\colon(0,\infty)\mapsto\mathbb{R}$ satisfying
$
  \widetilde{w}(\lambda)
  =
  \widetilde{w}(\lambda^{-1})
  \ 
  \text{for all }\lambda>0,
$
such that
$
  W(F)
  =
  \sum_i \widetilde{w}(\lambda_i)
  \ 
  \text{for all }F\in\mathrm{SL}(3).
 $

\end{enumerate}
\end{theorem}

\begin{proof}
Assume first that $W$ is Shield-invariant. Let
$
\lambda_i=e^{x_i},
\ 
x_1+x_2+x_3=0,
$
and define
$
  q(x)
  \coloneq
  w(e^x)-w(e^{-x}).
  \label{eq:def-q}
$
The function $q$ is continuous and odd.
Shield invariance gives
\[
\sum_i w(e^{x_i})
=
\sum_i w(e^{-x_i}),
\]
and hence
$
  q(x_1)+q(x_2)+q(x_3)=0
$ $
  \text{whenever }x_1+x_2+x_3=0.
$

Choose
$
x_1=x,
$ $
x_2=y,
$ $
x_3=-x-y.
$
Then
$
q(x)+q(y)+q(-x-y)=0.
$
Since $q$ is odd,
$
q(x+y)=q(x)+q(y).
$
Thus, $q$ satisfies the Cauchy functional equation \cite{Kuczma2009}. Since $q$ is
continuous, there exists $c\in\mathbb{R}$ such that
$
q(x)=cx.
$
Taking $\lambda=e^x$, we obtain
$
w(\lambda)-w(\lambda^{-1})
=
c\log\lambda.
$
This proves that \textnormal{(i)} implies \textnormal{(ii)}.

Assume now that \textnormal{(ii)} holds and define
$
  \widetilde{w}(\lambda)
  \coloneq
  w(\lambda)-\frac{c}{2}\log\lambda.
$
Then
$
\widetilde{w}(\lambda)-\widetilde{w}(\lambda^{-1})
=
w(\lambda)-w(\lambda^{-1})
-c\log\lambda
=0.
$
Therefore,
$
\widetilde{w}(\lambda)
=
\widetilde{w}(\lambda^{-1}).
$

Moreover, if $F\in\mathrm{SL}(3)$, then
\[
\begin{aligned}
\sum_{i=1}^3 \widetilde{w}(\lambda_i)
&=
\sum_{i=1}^3 w(\lambda_i)
-
\frac{c}{2}\sum_{i=1}^3 \log\lambda_i
=
\sum_{i=1}^3 w(\lambda_i)
-
\frac{c}{2}\log(\lambda_1\lambda_2\lambda_3)
=
\sum_{i=1}^3 w(\lambda_i).
\end{aligned}
\]
This proves that \textnormal{(ii)} implies \textnormal{(iii)}.
Finally, assume that \textnormal{(iii)} holds. Then
\[
\begin{aligned}
W(F^{-1})
&=
\sum_{i=1}^3 \widetilde{w}(\lambda_i^{-1})
=
\sum_{i=1}^3 \widetilde{w}(\lambda_i)
=
W(F).
\end{aligned}
\]
Thus, $W$ is Shield-invariant, proving that \textnormal{(iii)} implies
\textnormal{(i)}.
\end{proof}

\begin{remark}
\label{rem:canonical-symmetric-generator}
Under the assumptions of
Theorem~\ref{thm:VL-shield-characterization-SL3}, the function
$
  w_{\mathrm{sym}}(\lambda)
  \coloneq
  \frac12
  \mleft(
  w(\lambda)+w(\lambda^{-1})
  \mright)
  \label{eq:canonical-symmetric-generator}
$
is inversion-symmetric. Moreover, Shield invariance implies
$
\sum_i w_{\mathrm{sym}}(\lambda_i)
=
\sum_i w(\lambda_i)
$ $
\text{on }\mathrm{SL}(3).
$
Thus, every continuous Shield-invariant Valanis--Landel energy on
$\mathrm{SL}(3)$ admits a canonical inversion-symmetric scalar
representative.
\end{remark}

\begin{remark}
\label{rem:logarithmic-null-term}
The term $c\log\lambda$ is invisible in a Valanis--Landel energy restricted
to $\mathrm{SL}(3)$, since
\[
\sum_{i=1}^3 c\log\lambda_i
=
c\log(\lambda_1\lambda_2\lambda_3)
=
c\log\det F
=
0.
\]
Consequently, inversion symmetry of a particular scalar generator is not
strictly necessary. Nevertheless, every continuous Shield-invariant
Valanis--Landel energy on $\mathrm{SL}(3)$ admits an equivalent
inversion-symmetric generator.
\end{remark}

The preceding results reveal an essential difference between dimensions
two and three:
\begin{align}
\begin{array}{ll}
\mathrm{SL}(2):
&
\text{inversion symmetry of an objective and isotropic total energy is
automatic},
\\[2mm]
\mathrm{SL}(3):
&
\text{inversion symmetry is a genuine constitutive restriction}.
\end{array}
\label{eq:SL2-SL3-comparison}
\end{align}

For separable Valanis--Landel energies, this distinction is particularly
transparent. On $\mathrm{SL}(2)$, the principal stretches occur as the pair
$
(\lambda,\lambda^{-1}),
$
so that
$
w(\lambda)+w(\lambda^{-1})
$
is inversion-symmetric for every scalar function $w$.

On $\mathrm{SL}(3)$, the condition
$
\lambda_1\lambda_2\lambda_3=1
$
does not pair each principal stretch with its reciprocal. Hence, {inversion}
symmetry becomes a nontrivial restriction on the total energy.
 \end{footnotesize}

\end{document}